\documentclass{amsart}

\usepackage{amsfonts, amsmath, amssymb, amsthm}
\usepackage{mathrsfs}
\usepackage{array, multirow}
\usepackage{tikz}
\usetikzlibrary{positioning, scopes}

\newcommand{\st}{\,|\,}
\newcommand{\real}{\mathbb{R}}
\newcommand{\nat}{\mathbb{N}}

\newcommand{\integer}{\mathbb{Z}}

\newtheorem{theorem}{Theorem}[section]
\newtheorem{lemma}[theorem]{Lemma}
\newtheorem{corollary}[theorem]{Corollary}

\newtheorem{remark}[theorem]{Remark}
\newtheorem{example}[theorem]{Example}
\newtheorem{proposition}[theorem]{Proposition}

\begin{document}

\title[YANG INDICES OF GRASSMANNIANS]{The Yang Indices of Grassmannians}

\author[JAMES DIBBLE]{JAMES DIBBLE}
\address{Department of Mathematics and Statistics, University of Southern Maine, Portland, ME 04101}
\email{james.dibble@maine.edu}


\date{}

\begin{abstract}
An elementary combinatorial technique for computing lower bounds for the Yang indices of real Stiefel manifolds and oriented real Grassmannians is described. As a demonstration, it shown that the Yang index of $St(n,k)$, and consequently $G(n,k)$, is at least $n - k$. For odd $n$, the bound for $G(n,2)$ can be improved to $n-1$. These are combined with basic properties of the Yang index and Conner--Floyd index and coindex to compute the possible Yang indices of $St(n,k)$ and $G(n,k)$ for small $n$.
\end{abstract}

\maketitle

\section{Introduction}

Much about the topology of unoriented real Grassmannians is, by now, reasonably well understood. Their $\integer_2$-cohomology groups were computed by Borel \cite{Borel1953}, and, in many cases, their cup lengths are known (e.g., \cite{Bernstein1976}, \cite{Hiller1980}, \cite{Stong1982}). It seems that their integral cohomology groups are also known, although it is difficult to find a complete description in the literature. Oriented real Grassmannians are much less well understood, except in special cases, where their $\integer_2$-cohomology groups (e.g., \cite{KorbasRusin2016b}, \cite{Rusin2019}, \cite{BasuChakraborty2020}, \cite{MatszangoszWendt2024}, \cite{ColovicJovanovicPrvulovic2025}) and cup lengths (e.g., \cite{Korbas2015}, \cite{KorbasRusin2016a}, \cite{Rusin2018}) have been studied.

This paper introduces an elementary combinatorial technique for computing lower bounds for the Yang indices of real Stiefel manifolds and oriented real Grassmannians. The Yang index is a topological invariant introduced in \cite{Yang1954} to generalize the classical theorem of Borsuk--Ulam \cite{Borsuk1933} to more general spaces with free $\integer_2$-actions. In the case of the Stiefel manifold $St(n,k)$, the action considered here, $S$, negates the final entry of each list; for the Grassmannian $G(n,k)$, $T$ reverses the orientation of each hyperplane. As a demonstration of the technique, it is shown that the Yang index of $St(n,k)$, and consequently that of $G(n,k)$, is always at least $n - k$.

\begin{theorem}\label{main theorem}
	If $k$ and $n$ are integers satisfying $1 \leq k \leq n$, then
	\[
		\nu(G(n,k)) \geq \nu(St(n,k)) \geq n - k\textrm{.}
	\]
\end{theorem}

\noindent The theorem, which follows more immediately from basic results about Conner--Floyd index and coindex (cf. \cite{ConnerFloyd1962} or \cite{DaiLam1984}), is proved by explicitly constructing chains that represent $\integer_2$-invariant homology classes of high enough Yang index in an abstract simplicial complex whose faces are certain valid matrices with entries in $\{ \pm 1,\ldots,\pm n \}$. These matrices correspond to singular simplices into $St(n,k)$ of a relatively simple form. It implies a generalized Borsuk--Ulam theorem.

\begin{corollary}\label{borsuk--ulam generalization}
	Let $m$, $n$, and $k$ be integers satisfying $1 \leq m \leq n - k$. If $f : St(n,k) \to \real^m$ is continuous, then the set $\{ V \in St(n,k) \st f(S(V)) = f(V) \}$ is $S$-invariant and has Yang index at least $n - k - m$.
\end{corollary}

\noindent In particular, it follows that there does not exist an equivariant map from $St(n,k)$ into $S^{m-1}$ for $1 \leq m \leq n - k$. Similar results hold for $G(n,k)$.

When $k = 2$, it is possible to characterize the valid matrices combinatorially. By combining that observation with the additional flexibility that working in the Grassmannian affords, one is able to improve the above bound for $G(n,2)$ when $n$ is odd.

\begin{theorem}\label{k = 2 case}
	Let $n \geq 3$ be an odd integer. Then, $\nu(G(n,2)) \geq n - 1$.
\end{theorem}

\noindent This implies a corresponding Borsuk--Ulam result.

\begin{corollary}\label{k = 2 borsuk--ulam generalization}
	Let $n \geq 3$ be an odd integer and $1 \leq m \leq n - 1$ an integer. If $f : G(n,2) \to \real^m$ is continuous, then the set $\{ P \in G(n,k) \st f(T(P)) = f(P) \}$ is $T$-invariant and has Yang index at least $n - m - 1$.
\end{corollary}

\noindent Thus, for odd $n$, there does not exist an equivariant map from $G(n,2)$ into $S^{m-1}$ when $1 \leq m \leq n - 1$. This nonexistence was previously proved by Conner--Floyd \cite{ConnerFloyd1962}; in fact, it follows from their work that no such map exists from $St(n,2)$ for $n \neq 2, 4, \textrm{ or }8$. For odd $n$, this latter result also follows from recent work of Basu--Kundu \cite{BasuKundu2024} on the Faddell--Husseini indices of Stiefel manifolds. Like much of the earlier work discussed above, Basu--Kundu's results were proved using spectral sequences. The techniques of this paper, which were developed independently, are much more elementary in nature.\\

\noindent \textit{Organization of the paper.} Section 2 contains preliminaries, including background material on the Yang index in Subsection 2.1 and Stiefel manifolds and Grassmannians in Subsection 2.2. The notion of a valid matrix is developed in Subsection 2.3. The main theorems are proved in Section 3, with the general case in Subsection 3.1 and the $k = 2$ case in Subsection 3.2. Section 4 discusses the relationship between the Yang index and the Conner--Floyd index and coindex, develops some of their basic properties for Stiefel manifolds and Grassmannians, and computes the possible values of $\nu(St(n,k))$ and $\nu(G(n,k))$ for small values of $n$.

\section{Preliminaries}

\subsection{Yang index}

The results in this subsection are due to Yang in \cite{Yang1954}. For details and applications, one may consult \cite{Yang1954} and its companion paper \cite{Yang1955}. Excellent resources for broader perspectives on the Borsuk--Ulam theorem and more general fixed point theory are \cite{Matousek2003} and \cite{GranasDugundji2003}, respectively. Throughout this paper, all homology groups are computed with $\integer_2$ coefficients, and equivariance is always meant with respect to the appropriate $\integer_2$-actions.

A \textbf{$T$-space} is a pair $(X,T)$, where $X$ is a compact and Hausdorff topological space and $T$ is a fixed-point free involution on $X$. A $T$-space $(X,T)$ is \textbf{simplicial} if $X$ is a finite Euclidean simplicial complex whose simplices are permuted by $T$. In this case, an $m$-chain $c$ is \textbf{$T$-invariant} if $T(c) = c$. This is equivalent to $c = d + T(d)$ for some $m$-chain $d$. If $C_m(X,T)$ denotes the set of $T$-invariant $m$-chains, then the boundary map $\partial_{m+1}$ takes $C_{m+1}(X,T)$ into $C_m(X,T)$. Thus, one may define the \textbf{$T$-homology group} $H_m(X,T) = Z_m(X,T)/B_m(X,T)$, where $Z_m(X,T)$ is the kernel of $\partial_m$ and $B_m(X,T)$ is the image of $\partial_{m+1}$.  By \u{C}ech-Smith theory, these extend from simplicial to arbitrary $T$-spaces; in general, $H_m(X,T)$ is isomorphic to the $m$-th \u{C}ech homology group of $X/T$.

If $(X,T)$ is simplicial, the \textbf{index} $\nu(c)$ of a chain $c \in Z_0(X,T)$ that splits as $c = d + T(d)$ is defined to be $1$ if the cardinality of $d$ is odd and $0$ if even. The \textbf{index} of $c \in Z_m(X,T)$ is then defined inductively by $\nu(c) = \nu(\partial d)$ for any splitting $c = d + T(d)$. This defines a homomorphism $\nu : H_m(X,T) \to \integer_2$ for all $m$, which, in turn, generalizes to arbitrary $T$-spaces. The \textbf{Yang index} $\nu(X,T)$ of $(X,T)$ is the largest integer $m$ such that $\nu(H_m(X,T)) \neq \{ 0 \}$.

\begin{lemma}\label{index comparison}
	Let $(X,T)$ be a $T$-space. Then, the following hold:\\
	\textbf{(a)} $\nu(X,T)$ is the largest integer $m$ such that, for every $S$-space $(Y,S)$ and every equivariant map $f : X \to Y$, $f_*(H_m(X,T)) \neq \{ 0 \}$.\\
	\textbf{(b)} If $(Y,S)$ is an $S$-space and there exists an equivariant map from $X$ to $Y$, then $\nu(Y,S) \geq \nu(X,T)$.
\end{lemma}

\noindent The model $T$-space of dimension $n$ is the sphere $S^n$ with its usual antipodal action, which has Yang index $n$. More generally, if $(X,T)$ is a $T$-space, then $x,y \in X$ are \textbf{antipodal} if $y = T(x)$. The following is a substantial generalization of the classical theorem of Borsuk--Ulam.

\begin{theorem}\label{yang theorem}[Yang]
	Let $(X,T)$ be a $T$-space. If $\nu(X,T) = n$ and $m \leq n$, then every map $f : X \to \real^m$ takes two antipodal points to the same point. Moreover, the set $\{ x \in X \st f(x) = f(T(x)) \}$ is a $T$-space of Yang index at least $n - m$.
\end{theorem}

\noindent In particular, if $\nu(X,T) = n$, then there does not exist an equivariant map from $X$ into $S^{n-1}$. A variety of equivalent conclusions are stated in \cite{Yang1954}.

\subsection{Stiefel manifolds and Grassmannians}

The real Stiefel manifold $St(n,k)$ is the set of orthonormal $k$-tuples of vectors in $\real^n$, that is,
\[
	St(n,k) = \{ (v_1,\ldots,v_k) \st v_i \in \real^n \textrm{ and } v_i \cdot v_j = \delta_{i,j} \textrm{ for all } i,j \}\textrm{,}
\]
where $\cdot$ is the usual Euclidean dot product and $\delta_{i,j}$ is the Kronecker delta. This is naturally identified with the set of $n \times k$ matrices with orthonormal columns by the map that takes $(v_1,\ldots,v_k)$ to the matrix with columns $v_1,\ldots,v_k$. Thus, each element of $St(n,k)$ may be throught of as an $n \times k$ matrix $V$ satisfying $V^T V = I_k$, where $I_k$ is the $k \times k$ identity matrix. In this way, $St(n,k)$ is a smooth submanifold of $\real^{n+k}$ of dimension $nk - k(k+1)/2$. Much about the topology of Stiefel manifolds is discussed in \cite{James1976}.

The oriented real Grassmannian $G(n,k)$ is the set of oriented $k$-planes through the origin in $\real^n$, which for $k > 0$ is realized as the quotient
\[
	G(n,k) = St(n,k)/SO(k)\textrm{,}
\]
where the action of the special orthogonal group $SO(k)$ is on the right. It is a smooth manifold of dimension $k(n-k)$.

The map $S$ defined on $St(n,k)$ by $S(v_1,\ldots,v_k) = (v_1,\ldots,v_{k-1},-v_k)$ is a free involution, as is the map $T$ on $G(n,k)$ that reverses the orientation of each $k$-plane. Thus, $(St(n,k),S)$ and $(G(n,k),T)$ are $S$- and $T$-spaces, respectively. Throughout the rest of the paper, the $\integer_2$-actions on $St(n,k)$ and $G(n,k)$ will always be $S$ and $T$, respectively, and will be suppressed in the notation. Since the quotient map from $St(n,k)$ to $G(n,k)$ is equivariant, it follows from Lemma \ref{index comparison} that $\nu(G(n,k)) \geq \nu(St(n,k))$. The quotient map from $G(n,k)$ into the unoriented Grassmannian $\tilde{G}(n,k) = G(n,k)/T$ is a double cover, and each $T$-homology group $H_m(G(n,k),T)$ is isomorphic to the singular homology group $H_m(\tilde{G}(n,k))$, since \u{C}ech and singular homology are equivalent for manifolds (cf. \cite{Mardesic1959}).

\subsection{Valid matrices}

Let $m$, $n$, and $k$ be integers, where $m$ is nonnegative, $n$ and $k$ are positive, and $k \leq n$. Write the standard basis vectors of $\real^{m+1}$ as $e_1,\ldots,e_{m+1}$, where $e_i$ has $1$ in the $i$-th component and zeros everywhere else. Denote the standard $m$-simplex in $\real^{m+1}$ by
\[
	\Delta^m = \Big\{ \sum_{i=1}^{m+1} t_i e_i \st t_i \geq 0 \textrm{ for each } i \textrm{ and } \sum_{i=1}^{m+1} t_i = 1 \Big\}\textrm{.}
\]
For any list $V = (v_1,\ldots,v_{m+1})$ of unit vectors in $\real^n$, let $\ell_V : \Delta^m \to \real^n$ be the restriction to $\Delta^m$ of the linear map that takes each $e_i$ to $v_i$. For any list $\mathscr{V} = (V_1,\ldots,V_k)$ of such lists, define a map $\ell_{\mathscr{V}} : \Delta^m \to (\real^n)^k$ by $\ell_\mathscr{V} = (\ell_{V_1},\ldots,\ell_{V_k})$. Applying the Gram--Schmidt process to $\ell_{\mathscr{V}}$ produces a singular $m$-simplex $L_{\mathscr{V}} : \Delta^m \to St(n,k)$ exactly when $\ell_{\mathscr{V}}(T)$ is linearly independent for each $T \in \Delta^m$. When this happens, the list $\mathscr{V}$ will be called \textbf{valid}. If $\mathscr{V}$ is valid, any sublist of $\mathscr{V}$ is also valid.

Clearly, no list can be valid if it contains component lists $V_1$ and $V_2$ such that, for some $i$, $V_1^i = -V_2^i$ (here, $V^i$ denotes the $i$-th entry of $V$). More generally, a list is valid if and only if, for each $b = (b_1,\ldots,b_k) \in \real^k$, the origin is not a convex combination of the vectors $\sum_{i=1}^k b_k V_k^1, \ldots, \sum_{i=1}^k b_k V_k^{m+1}$. In other words, for the \textbf{induced matrix}
\[
	M(\mathscr{V},b) = \Big[ \sum_{i=1}^k b_k V_k^1 \cdots \sum_{i=1}^k b_k V_k^{m+1} \Big]
\]
of the pair $(\mathscr{V},b)$, one has the following characterization of valid lists.

\begin{lemma}\label{valid list}
	The list $\mathscr{V} = (V_1,\ldots,V_k)$ of lists of $m+1$ unit vectors in $\real^n$ is valid if and only if, for each $b \in \real^k$, the convex hull of the columns of $M(\mathscr{V},b)$ does not contain the origin.
\end{lemma}

\noindent By Farkas' lemma \cite{Farkas1902}, this is equivalent to, for each $b \in \real^k$, the existence of a hyperplane in $\real^n$ separating the origin from the columns of $M(\mathscr{V},b)$.

For each $k \times (m+1)$ matrix $A$ with entries in $\{ \pm 1,\ldots, \pm n\}$, i.e.,
\[
	A = \left[ \begin{array}{ccc} s_{1,1} I_{1,1} & \cdots & s_{1,m+1} I_{1,m+1} \\ \vdots & \phantom{ } & \vdots \\ s_{k,1} I_{k,1} & \cdots & s_{k,{m+1}} I_{k,{m+1}} \end{array} \right]\textrm{,} 
\]
where $s_{i,j} \in \{ 1, -1 \}$ and $I_{i,j} \in \{ 1, \ldots, n \}$, let $A_i = (s_{i,1} e_{I_{i,1}}, \ldots, s_{i,m+1} e_{I_{i,m+1}})$, and write $\mathscr{A} = (A_1,\ldots,A_k)$. Overloading notation, define the \textbf{induced matrix} $M(A,b)$ of the pair $(A,b)$, where $b \in \real^k$, to be the induced matrix of $(\mathscr{A},b)$. When $\mathscr{A}$ is a valid list with distinct columns, the matrix $A$ will be called \textbf{valid}. In this case, $L_{\mathscr{A}}$ is a singular $m$-simplex in $St(n,k)$, the vertices of which are frames whose component vectors lie on the coordinate axes in $\real^n$.

\section{Proof of the main theorems}

\subsection{General approach}

Let $n$ and $k$ be positive integers such that $k \leq n$. Create an abstract simplicial complex $X(n,k)$ whose elements are the sets of columns of valid $k \times (m+1)$ matrices with entries in $\{ \pm 1,\ldots, \pm n \}$, i.e.,
\[
	X(n,k) = \big\{ \{ a_1,\ldots,a_{m+1} \} \st a_i \in \{ \pm 1,\ldots,\pm n \}^k \textrm{ and } [a_1 \cdots a_{m+1}] \textrm{ is a valid matrix} \big\}\textrm{.}
\]
Since valid matrices have distinct columns, the identification of valid matrices with sets of column vectors is one-to-one up to permutation of the columns. Let $\mathscr{T}(n,k)$ be the total space of $X(n,k)$. That is,
\[
	\mathscr{T}(n,k) = \coprod_{\{a_1,\ldots,a_m\} \in X(n,k)} \Delta(a_1,\ldots,a_m) \,/ \sim\textrm{,}
\]
where each $\Delta(a_1,\ldots,a_{m+1})$ is a copy of the standard $m$-simplex whose vertices are identified with $a_1,\ldots,a_{m+1}$ in some order and $\sim$ identifies faces defined by the same vertex sets.

Define a fixed-point free involution $\tau$ on $\mathscr{T}(n,k)$ by $\tau(\sum_{i=1}^{m+1} t_i a_i) = \sum_{i=1}^{m+1} t_i \tilde{a}_i$, where $\tilde{a}_i$ is obtained from $a_i$ by multiplying the last component by $-1$. Since there are only finitely many valid matrices with $k$ rows and entries in $\{ \pm 1,\ldots, \pm n\}$, $X(n,k)$ is a finite complex, $\mathscr{T}(n,k)$ is compact, and $(\mathscr{T}(n,k),\tau)$ is a $\tau$-space.

For each $\{a_1,\ldots,a_{m+1}\} \in X(n,k)$, there is a unique list $\mathscr{A}$ containing the $a_i$ in the order they correspond to the standard basis vectors within $\Delta(a_1,\ldots,a_{m+1})$. The map $L_\mathscr{A}$ is a singular $m$-simplex into $St(n,k)$, and one may define an inclusion $\iota_{n,k}$ from $\mathscr{T}(n,k)$ into $St(n,k)$ that agrees with the appropriate $L_\mathscr{A}$ on each face. Since $\iota_{n,k}$ is equivariant, the Yang indices of the spaces satisfy the following bounds.

\begin{lemma}
	For each $1 \leq k \leq n$, $\nu(G(n,k)) \geq \nu(St(n,k)) \geq \nu(\mathscr{T}(n,k))$.
\end{lemma}

\noindent Define $\tilde{\mathscr{T}}(n,k)$ to be the quotient of $\mathscr{T}(n,k)$ that identifies faces along which $\iota_{n,k}$ spans the same oriented subspace of $\real^n$ at each point. Then, $\tau$ descends to a fixed-point free involution $\tilde{\tau}$ on $\tilde{\mathscr{T}}(n,k)$, $(\tilde{\mathscr{T}}(n,k),\tilde{\tau})$ is a $\tilde{\tau}$-space, and $\iota_{n,k}$ descends to an equivariant map $\tilde{\iota}_{n,k}$ from $\tilde{\mathscr{T}}(n,k)$ into $G(n,k)$. Thus, the following index bounds also hold.

\begin{lemma}
	For each $1 \leq k \leq n$, $\nu(G(n,k)) \geq \nu(\tilde{\mathscr{T}}(n,k)) \geq \nu(\mathscr{T}(n,k))$.
\end{lemma}

\noindent Each valid matrix corresponds to an ordering of the vertices of some face of $\mathscr{T}(n,k)$. Thus, the simplicial homology groups of $\mathscr{T}(n,k)$, including its $\tau$-homology groups, may be computed using chains of valid matrices. On these chains, the boundary operator has a simple expression: For a valid matrix $A = [a_1 \, \cdots \, a_{m+1}]$ with columns $a_1,\ldots,a_{m+1} \in \real^k$,
\[
	\partial A = \sum_{j=1}^{m+1} [a_1 \, \cdots \, \hat{a_j} \, \cdots \, a_{m+1}]\textrm{,}
\]
where $\hat{a_j}$ indicates the $j$-th column of $A$ has been removed.

The Yang index of $\mathscr{T}(n,k)$, and consequently that of $St(n,k)$, is at least $m$ whenever there exists a $\tau$-invariant chain $c$ of valid $k \times (m+1)$ matrices such that $\partial c = 0$ and $\nu(c) = 1$. The same is true in $\tilde{\mathscr{T}}(n,k)$, with the understanding that the chains are on equivalence classes of valid matrices, where two valid matrices are equivalent if they determine the same singular simplex in $G(n,k)$. The following characterization is a basic exercise in linear algebra.

\begin{lemma}\label{valid matrices}
	Let $A$ and $B$ be valid matrices corresponding to the lists $\mathscr{A}$ and $\mathscr{B}$, respectively. Then, $L_\mathscr{A}$ and $L_\mathscr{B}$ descend to the same singular simplex in $G(n,k)$ exactly when $A$ can be turned into $B$ via an even number of the following elementary row operations:\\
(1) Multiply any row by $-1$.\\
(2) Interchange any two rows.
\end{lemma}

\noindent If either of the identifications in Lemma \ref{valid matrices} is required to show that $\partial c = 0$ or $\nu(c) = 1$, then the corresponding index bound might not hold on $\mathscr{T}(n,k)$, but it does on $\tilde{\mathscr{T}}(n,k)$ and, consequently, on $G(n,k)$. In this way, establishing lower bounds for $\nu(St(n,k))$ and $\nu(G(n,k))$ becomes a combinatorial exercise in finding appropriate chains of valid matrices.

In a chain of valid matrices, the symbol $\pm$ will be used to indicate summation over both possible signs of the entry in that spot. For example,
\[
	\left[ \begin{array}{cc} 1 & 1 \\ \pm 2 & \pm 3 \end{array} \right] = \sum_{s_1,s_2 \in \{-1,1\}} \left[ \begin{array}{cc} 1 & 1 \\ s_1 2 & s_2 3 \end{array} \right]\textrm{.}
\]
If a matrix contains $K$ distinct entries with $\pm$ in them, it represents the sum of $2^K$ individual matrices.

\begin{proof}[Proof of Theorem \ref{main theorem}]
	The $\tau$-invariant chain of valid $k \times (n-k+1)$ matrices
	\[
		\left[ \begin{array}{ccc} 1 & \cdots & 1 \\ \vdots & \phantom{a} & \vdots \\ k - 1 & \cdots & k - 1 \\ \pm k & \cdots & \pm n \end{array} \right]
	\]
	has vanishing boundary and nonzero Yang index.
\end{proof}

\begin{remark}
	The proof of Theorem \ref{main theorem} roughly corresponds to the fact that, once an orthonormal $(k-1)$-tuple $(v_1,\ldots,v_{k-1})$ in $\real^n$ is fixed, the map $v \mapsto (v_1,\ldots,v_{k-1},v)$ from the unit $(n-k)$-sphere in the orthogonal complement of $\{ v_1,\ldots,v_{k-1} \}$ is equivariant. In the above, $v_i = e_i$.
\end{remark}

\subsection{The case $k = 2$}

For large values of $k$ and $m$, it is not clear how to identify valid matrices at a glance. If $A$ is a $k \times (m+1)$ matrix with entries in $\{ \pm 1,\ldots,\pm n \}$ and associated list $\mathscr{A}$, $\det(\ell_{\mathscr{A}})$ is a polynomial on $\Delta^m$ that vanishes exactly when $A$ is invalid. In principle, validity may therefore be checked using numerical techniques, but those are generally inefficient. Fortunately, it will be shown that, for $k = 2$, the valid matrices admit a simple, combinatorial characterization.

If $A = [a_{i,j}]$ is a $k \times (m+1)$ matrix with entries in $\{ \pm 1, \ldots, \pm n \}$, a \textbf{circuit} in $A$ is a choice of distinct rows $i_0$ and $i_1$ and distinct columns $j_1,\ldots,j_p$ such that $a_{i_1,j_{\alpha}} = a_{i_0,j_{\alpha + 1}}$ for each $1 \leq \alpha \leq p - 1$ and $a_{i_1,j_p} = a_{i_0,j_1}$. An \textbf{anti-circuit} is the same, except $a_{i_1,j_{\alpha}} = -a_{i_0,j_{\alpha + 1}}$ and $a_{i_1,j_p} = -a_{i_0,j_1}$. Clearly, a matrix contains a circuit if and only if scaling some row by $-1$ creates an anti-circuit, and vice-versa. In either case, $p$ is the \textbf{length} of the circuit. Note that a circuit of length one exists in a matrix exactly when one of its columns contains the same entry twice; for an anti-circuit, it's when one column contains an entry and its negative.

\begin{lemma}\label{combinatorial characterization}
	A $2 \times (m+1)$ matrix $A$ with entries in $\{ \pm 1,\ldots, \pm n\}$ and distinct columns is invalid if and only if one of the following holds:\\
	\indent (1) One of its rows contains both $\alpha$ and $-\alpha$ for some $\alpha \in \{ 1,\ldots, n \}$.\\
	\indent (2) It contains a circuit or an anti-circuit.
\end{lemma}

\begin{proof}
	It is clear that $A$ is invalid whenever it satisfies (1). Suppose $A$ satisfies (2). If it contains an anti-circuit in columns $j_1,\ldots,j_p$, set $a = (1,1)$, and let $c \in \real^{m+1}$ contain $1/p$ in components $j_1,\ldots,j_p$ and zeros everywhere else. Then, the induced matrix $M = M(A,a)$ satisfies $Mc = 0$, so $A$ is invalid. If it contains a circuit, then the same is true with $a = (1,-1)$. 

	Conversely, suppose $A$ is invalid and doesn't satisfy (1). Choose a nonzero $b = (b_1,b_2) \in \real^2$ such that the origin is in the convex hull of the columns of the induced matrix $M = M(A,b)$. Note that neither $b_1$ nor $b_2$ can be zero, as otherwise (1) would hold. If $A$ doesn't contain a circuit or anti-circuit of length one, then no column of $M$ can be the zero vector. In that case, the first column of $M$ is nonzero in row $|a_{2,1}|$, and $M$ must contain an entry in row $|a_{2,1}|$ with sign opposite that of its $(1,|a_{2,2}|)$-th entry. Suppose such an entry lies in column $j_{p_1}$. The corresponding column of $A$ must have $\pm a_{2,1}$ as one of its two entries. Since (1) doesn't hold, $a_{2,j_{p_1}} \neq -a_{2,1}$. If $a_{2,j_{p_1}} = a_{2,1}$, then the corresponding entries of $M$ have the same sign, which contradicts the definition of $j_{p_1}$. Therefore, $a_{1,j_{p_1}} = s_1 a_{2,1}$, where $s_1 \in \{ -1, 1 \}$. By construction, $s_1$ must have sign opposite that of $b_1b_2$. Repeating this process a total of at most $\min\{m,n - 1\}$ times, alternating between the two rows of $A$, produces columns $1,j_{p_1},\ldots,j_{p_K}$ that contain a circuit when $b_1$ and $b_2$ have opposite signs and an anti-circuit when they have the same sign.
\end{proof}

\begin{remark}
	Any matrix that satisfies either (1) or (2) in Lemma \ref{combinatorial characterization} must necessarily be invalid, as is any matrix that contains an invalid submatrix with distinct columns. These observations can be used to improve algorithms to identify valid matrices when $k > 2$.
\end{remark}

\noindent It follows from Lemma \ref{combinatorial characterization} that the Yang index of $St(n,2)$ is at least $m$ whenever there exists a $\tau$-invariant chain of $2 \times (m+1)$ matrices with vanishing boundary and nonzero Yang index and whose matrices avoid (1) and (2). For $G(n,2)$, the existence of such a $\tilde{\tau}$-invariant chain suffices.

	Expanding on the $\pm$ notation from the previous subsection, in a chain of valid matrices, the symbol $\mp$ will be used to indicate that the signs of an entry vary but, in order for each matrix in the chain to remain valid, are determined by the signs of the entries with $\pm$ in them. This is necessary to avoid circuits and anti-circuits when one column is a ``twist'' of another column, i.e., contains entries of the same magnitudes but in a different order. For example,
\[
	\left[ \begin{array}{cc} 1 & \pm 2 \\ \pm 2 & \mp 1 \end{array} \right] = \left[ \begin{array}{cc} 1 & 2 \\ 2 & -1 \end{array} \right] + \left[ \begin{array}{cc} 1 & 2 \\ -2 & 1 \end{array} \right] + \left[ \begin{array}{cc} 1 & -2 \\ 2 & 1 \end{array} \right] + \left[ \begin{array}{cc} 1 & -2 \\ -2 & -1 \end{array} \right]\textrm{.}
\]
The matrices $\left[ \begin{array}{cc} 1 & 2 \\ 2 & 1 \end{array} \right]$, $\left[ \begin{array}{cc} 1 & -2 \\ -2 & 1 \end{array} \right]$, $\left[ \begin{array}{cc} 1 & 2 \\ -2 & -1 \end{array} \right]$, and $\left[ \begin{array}{cc} 1 & -2 \\ 2 & -1 \end{array} \right]$ are excluded from the chain, since the first two contain circuits and the last two contain anti-circuits. The $\mp$ notation also helps to avoid placing a number and its negative in the same row. For example,
\[
	\left[ \begin{array}{cc} 1 &  2 \\ \pm 3 & \mp 3 \end{array} \right] = \left[ \begin{array}{cc} 1 & 2 \\ 3 & 3 \end{array} \right] + \left[ \begin{array}{cc} 1 & 2 \\ -3 & -3 \end{array} \right]\textrm{.}
\]
It is still the case that a matrix containing $K$ entries with $\pm$ in them represents the sum of $2^K$ individual matrices.

\begin{example}\label{G(3,2)}
	Since $G(3,2) \cong S^2$, $\nu(G(3,2)) = 2$. On $S^2$, the chain $c$ consisting of the eight $2$-simplices $[\pm e_1 \, \pm e_2 \pm e_3]$, which represents the fundamental class, is antipodal and satisfies $\partial(c) = 0$ and $\nu(c) = 1$. However, none of those simplices corresponds to a valid matrix. The issue is holonomy: Tracing the third vector in a frame around a loop will always return the first two to their original subspace, but typically not to the original vectors.

For example, following the edges of $[e_3 \, -e_2 \, e_1]$ in order, its vertices may be identified with$\left[ \begin{array}{c} 1 \\ 2 \end{array} \right]$, $\left[ \begin{array}{c} 1 \\ 3 \end{array} \right]$, and $\left[ \begin{array}{c} 2 \\ 3 \end{array} \right]$, respectively, but returning along the final edge ends at $\left[ \begin{array}{c} 2 \\ -1 \end{array} \right]$. While $\left[ \begin{array}{c} 1 \\ 2 \end{array} \right]$ and $\left[ \begin{array}{c} 2 \\ -1 \end{array} \right]$ determine the same oriented plane in $G(3,2)$, they're different elements of $St(3,2)$. This can be rectified by stretching out an edge along which the simplex is constant in the Grassmannian and splitting in half, as in the following diagram:

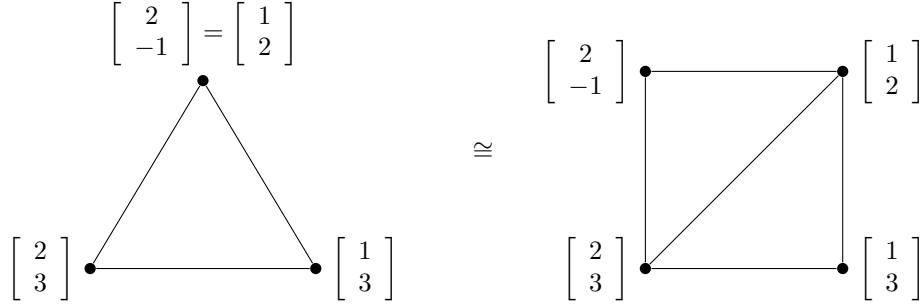
\vspace{3pt}

\begin{tikzpicture}
	\begin{scope}[local bounding box=triangle_box]
		\node (ta) at (0,0) [circle, fill=black, inner sep=1.5pt, label=left:{$\left[ \begin{array}{c} 2 \\ 3 \end{array} \right]$}] {};
		\node (tb) at (3,0) [circle, fill=black, inner sep=1.5pt, label=right:{$\left[ \begin{array}{c} 1 \\ 3 \end{array} \right]$}] {};
		\node (tc) at (1.5,2.5) [circle, fill=black, inner sep=1.5pt, label=above:{$\left[ \begin{array}{c} 2 \\ -1 \end{array} \right] = \left[ \begin{array}{c} 1 \\ 2 \end{array} \right]$}] {};

		\draw (ta) -- (tb) -- (tc) -- (ta);
	\end{scope}

	\node (cong) at ([xshift=1cm] triangle_box.east) {$\cong$};

	\node (ra) at ([xshift=4.3cm] tb.east) [circle, fill=black, inner sep=1.5pt, label=left:{$\left[ \begin{array}{c} 2 \\ 3 \end{array} \right]$}] {};
	\node (rb) [above=2.46cm of ra, circle, fill=black, inner sep=1.5pt, label=left:{$\left[ \begin{array}{c} 2 \\ -1 \end{array} \right]$}] {};
	\node (rc) [right=2.46cm of rb, circle, fill=black, inner sep=1.5pt, label=right:{$\left[ \begin{array}{c} 1 \\ 2 \end{array} \right]$}] {};
	\node (rd) [below=2.46cm of rc, circle, fill=black, inner sep=1.5pt, label=right:{$\left[ \begin{array}{c} 1 \\ 3 \end{array} \right]$}] {};

	\draw (ra) -- (rb) -- (rc) -- (rd) -- (ra);
	\draw (ra) -- (rc);
\end{tikzpicture}

\vspace{3pt}

\noindent In other words, $[e_3 \, -e_2 \, e_1]$ corresponds to the sum $\left[ \begin{array}{ccc} 1 & 1 & 2 \\ 2 & 3 & 3 \end{array} \right] +\left[ \begin{array}{ccc} 1 & 2 & 2 \\ 2 & -1 & 3 \end{array} \right]$. Repeating this over all eight simplices, one finds that the fundamental class of $G(3,2)$ is represented by the chain
\[
	\left[ \begin{array}{ccc} 1 & 1 & \pm 2 \\ \pm 2 & \pm 3 & \mp 3 \end{array} \right] +\left[ \begin{array}{ccc} 1 & \pm 2 & \mp 2 \\ \pm 2 & \mp 1 & \pm 3 \end{array} \right]
\]
of sixteen valid matrices.
\end{example}

\begin{proof}[Proof of Theorem \ref{k = 2 case}(a)]
	Let $n \geq 3$ be odd. None of the $2 \times n$ matrices in the $\tilde{\tau}$-invariant chain
	\begin{align*}
		c =&\sum_{i = 3}^{n-1} \left[ \begin{array}{cccccccc} 1 & \cdots & 1 & 1 & \pm 2 & \cdots & \mp 2 \\ \pm 2 & \cdots & \pm i & \pm n & \mp i & \cdots & \pm (n - 1) \end{array} \right]\\
		&+ \left[ \begin{array}{cccccc} 1 & 1 & \pm 2 & \mp 2 & \cdots & \mp 2 \\ \pm 2 & \pm n & \mp 1 & \pm 3 & \cdots & \pm (n-1) \end{array} \right]\\
		&+ \left[ \begin{array}{cccccc} 1 & 1 & \pm 2 & \mp 2 & \cdots & \mp 2 \\ \mp (n-1) & \pm n & \pm 1 & \pm 3 & \cdots & \pm (n-1) \end{array} \right]\\
		&+ \left[ \begin{array}{ccccccc} 1 & 1 & \pm 2 & \mp 2 & \cdots & \mp 2 & \pm 2 \\ \pm (n-1) & \mp n & \pm 1 & \pm 3 & \cdots & \pm (n-2) & \pm n\end{array} \right]\\
		&+ \left[ \begin{array}{ccccc} 1 & \pm 2 & \mp 2 & \cdots & \mp 2 \\ \mp (n-1) & \pm 1 & \pm 3 & \cdots & \pm n\end{array} \right]\\
	\end{align*}
contains a circuit, an anti-circuit, or a row containing a number and its negative, so they are all valid. It is nearly routine to verify that $\partial c = 0$, although one must be careful with the last summand; the chains
\[
	\left[ \begin{array}{ccccccc} \pm 2 & \mp 2 & \cdots & \mp 2 & \mp 2 & \mp 2 \\ \pm 1 & \pm 3 & \cdots & \pm (n-2) & n-1 & \pm n \end{array} \right]
\]
and
\[
	\left[ \begin{array}{ccccccc} \pm 2 & \mp 2 & \cdots & \mp 2 & \mp 2 & \mp 2 \\ \pm 1 & \pm 3 & \cdots & \pm (n-2) & -(n-1) & \pm n \end{array} \right]
\]
must cancel over $\integer_2$, which is true in $\tilde{\mathscr{T}}(n,2)$ but not in $\mathscr{T}(n,2)$.

That $\nu(c) = 1$ follows from an inductive argument directly according to Yang's definition. Write $c_n = c$. Splitting $c_n = d_n + T(d_n)$, where
\begin{align*}
		d_n =&\sum_{i = 3}^{n-1} \left[ \begin{array}{cccccccc} 1 & \cdots & 1 & 1 & 1 & \pm 2 & \cdots & \mp 2 \\ \pm 2 & \cdots & \pm (i-1) & \mp i & n & \pm i & \cdots & \pm (n - 1) \end{array} \right]\\
		&+ \left[ \begin{array}{cccccc} 1 & 1 & \pm 2 & \mp 2 & \cdots & \mp 2 \\ \pm 2 & n & \mp 1 & \pm 3 & \cdots & \pm (n-1) \end{array} \right]\\
		&+ \left[ \begin{array}{cccccc} 1 & 1 & \pm 2 & \mp 2 & \cdots & \mp 2 \\ \mp (n-1) & n & \pm 1 & \pm 3 & \cdots & \pm (n-1) \end{array} \right]\\
		&+ \left[ \begin{array}{ccccccc} 1 & 1 & \pm 2 & \mp 2 & \cdots & \mp 2 & \pm 2 \\ \pm (n-1) & n & \pm 1 & \pm 3 & \cdots & \pm (n-2) & n\end{array} \right]\\
		&+ \left[ \begin{array}{cccccc} 1 & \pm 2 & \mp 2 & \cdots & \mp 2 & \mp 2 \\ \mp (n-1) & \pm 1 & \pm 3 & \cdots & \pm(n-1) & n\end{array} \right]\textrm{,}
\end{align*}
one computes that
\begin{align*}
	\partial(d_n) =&\sum_{i = 3}^{n-1} \left[ \begin{array}{cccccccc} 1 & \cdots & 1 & 1 & \pm 2 & \cdots & \mp 2 \\ \pm 2 & \cdots & \pm (i-1) & \mp i & \pm i & \cdots & \pm (n - 1) \end{array} \right]\\
			&+ \left[ \begin{array}{ccccc} 1 & \pm 2 & \mp 2 & \cdots & \mp 2 \\ \pm 2 & \mp 1 & \pm 3 & \cdots & \pm (n - 1) \end{array} \right]\\
			&+ \left[ \begin{array}{ccccc} \pm 2 & \mp 2 & \cdots & \mp 2 & \mp 2 \\ \pm 1 & \pm 3 & \cdots & \pm (n - 1) & n \end{array} \right]\textrm{.}
\end{align*}
Thus, $\nu(c) = \nu(c_{n-1}) + \nu \Big( \left[ \begin{array}{ccccc} \pm 2 & \mp 2 & \cdots & \mp 2 & \mp 2 \\ \pm 1 & \pm 3 & \cdots & \pm (n - 1) & n \end{array} \right] \Big)$, where $c_{n-1}$ is the chain in the first two lines on the right-hand side of the preceding equality. Continuing to split the $c_i$ around the terms of largest magnitude, one eventually finds that
\[
	\nu(c) = \nu \Big( \left[ \begin{array}{cc} 1 & \pm 2 \\ \pm 2 & \mp 1 \end{array} \right] \Big) + \sum_{i=3}^n \nu \Big( \left[ \begin{array}{ccccc} \pm 2 & \mp 2 & \cdots & \mp 2 & \mp 2 \\ \pm 1 & \pm 3 & \cdots & \pm (i-1) & i \end{array} \right] \Big)\textrm{.}
\]
Since
\[
	\nu \Big( \left[ \begin{array}{cc} 1 & \pm 2 \\ \pm 2 & \mp 1 \end{array} \right] \Big) = 0
\]
and
\[
	\nu \Big( \left[ \begin{array}{ccccc} \pm 2 & \mp 2 & \cdots & \mp 2 & \mp 2 \\ \pm 1 & \pm 3 & \cdots & \pm (i-1) & i \end{array} \right] \Big) = \nu \Big( \left[ \begin{array}{cccc} 2 & 2 & \cdots & 2 \\ \pm 1 & \pm 3 & \cdots & \pm i \end{array} \right] \Big) = 1\textrm{,}
\]
it follows that $\nu(c) = \sum_{i=3}^n 1 = 1$, as $n$ is odd.
\end{proof}

\section{Relation to Conner--Floyd index and coindex}

The exact values of the Yang indices of $St(n,k)$ and $G(n,k)$ remains an open question in all but a handful of cases. This is closely related to the computation of their Conner--Floyd indices and coindices \cite{ConnerFloyd1960}. The \textbf{Conner--Floyd index} $\mathrm{ind}(X,T)$ and \textbf{coindex} $\mathrm{coind}(X,T)$ of a $T$-space $(X,T)$ are defined by
\[
	\mathrm{ind}(X,T) = \max \{ n \geq 0 \st \textrm{there exists a } \integer_2 \textrm{-equivariant map } S^n \to X \}
\]
and
\[
	\mathrm{coind}(X,T) = \min \{ n \geq 0 \st \textrm{there exists a } \integer_2 \textrm{-equivariant map } X \to S^n \}\textrm{,}
\]
where the maximum and minimum are taken within $\nat \cup \{ \infty \}$. By Lemma \ref{index comparison}(b) and Theorem \ref{yang theorem},
\[
	\mathrm{ind}(X,T) \leq \nu(X,T) \leq \mathrm{coind}(X,T)\textrm{.}
\]
It is shown in \cite{ConnerFloyd1960} that $\mathrm{coind}(X,T) \leq \mathrm{dim}(X/T)$, where $\dim$ is covering dimension. If $(X,T)$ maps equivariantly into $(Y,S)$, then $\mathrm{ind}(X,T) \leq \mathrm{ind}(Y,S)$ and $\mathrm{coind}(X,T) \leq \mathrm{coind}(Y,S)$. Yang \cite{Yang1955} gave an example of a space with $\nu(X,T) < \mathrm{coind}(X,T)$.

For $St(n,1) = S^{n-1}$, the Conner--Floyd and Yang indices and coindex are all $n - 1$. It was proved in \cite{DaiLam1984} that $\mathrm{ind}(St(n,2)) = n - 2$ and in \cite{ConnerFloyd1962} (with a simpler argument given in \cite{DaiLam1984}) that
\[
	\mathrm{coind}(St(n,2)) = \left\{ \begin{array}{ccc} n - 1 & \textrm{if} & n \neq 2, 4, \textrm{ or } 8\\ n - 2 & \textrm{if} & n = 2, 4, \textrm{ or } 8 \end{array} \right. \textrm{.}
\]
It follows that the indices and coindex of $St(n,2)$ are all $n - 2$ for $n = 2,4, \textrm{ and } 8$ and either $n-2$ or $n-1$ otherwise. For larger values of $k$, not much appears to be known in the literature. Some basic properties are developed here.

\begin{proposition}\label{stiefel monotonicity}
	Let $1 \leq k \leq n$. Then, the following hold:\\
	\textbf{(a)} For each $n$, the Conner--Floyd and Yang indices and coindex of $St(n,k)$ are nonincreasing in $k$.\\
	\textbf{(b)} For each $k$, the Conner--Floyd and Yang indices and coindex of $St(n,k)$ are nondecreasing in $n$.\\
	\textbf{(c)} For each $n$ and $k$, the Conner--Floyd and Yang indices and coindex of $St(n+\ell,k+\ell)$ are nondecreasing in $\ell$.
\end{proposition}

\begin{proof} \textbf{(a)} The map from $St(n,k+\ell)$ into $St(n,k)$ that projects onto the last $k$ components is equivariant.\\
\noindent \textbf{(b)} If $\real^n$ is identified with any $n$-dimensional subspace of $\real^{n+m}$, then the inclusion map takes $St(n,k)$ into $St(n+m,k)$ equivariantly.\\
\noindent \textbf{(c)} For each $V \in St(n + \ell,\ell)$, once an orthonormal basis for $V^\perp \cong \real^n$ is fixed, $W \mapsto (V,W)$ defines an equivariant map from $St(n,k)$ into $St(n + \ell,k + \ell)$.
\end{proof}

\noindent It follows from (a) that
\[
	n - k \leq \mathrm{ind}(St(n,k)) \leq \nu(St(n,k)) \leq \mathrm{coind}(St(n,k)) \leq \mathrm{coind}(St(n,2))
\]
for all $n \geq 2$ and from (b) and (c) that the indices and coindex of $St(n+m,k+\ell)$ are all, respectively, at least as large as those of $St(n,k)$ whenever $0 \leq \ell \leq m$.

Since $G(n,k)$ is an equivariant quotient of $St(n,k)$, the indices and coindex of $G(n,k)$ are all at least as large as those of $St(n,k)$. The dimension bound gives that they are all at most $k(n-k)$.

\begin{proposition}\label{grassmannian monotonicity}
	Let $0 \leq k \leq n$. Then, the following hold:\\
	\textbf{(a)} For each $n$ and $k$, the Conner--Floyd and Yang indices and coindex of $G(n,k)$ agree, respectively, with those of $G(n,n-k)$.\\
	\textbf{(b)} For each $k$, the Conner--Floyd and Yang indices and coindex of $G(n,k)$ are nondecreasing in $n$.\\
	\textbf{(c)} For each $n$ and $k$, the Conner--Floyd and Yang indices and coindex of $G(n+\ell,k+\ell)$ are nondecreasing in $\ell$.
\end{proposition}

\begin{proof}
\noindent \textbf{(a)} Taking orthogonal complements (respecting orientation) defines an equivariant homemorphism between $G(n,k)$ and $G(n,n-k)$.\\
\noindent \textbf{(b)} As in the proof of Proposition \ref{stiefel monotonicity}(b), $\real^n$ may be identified with any $n$-dimensional subspace of $\real^{n+m}$.\\
\noindent \textbf{(c)} For each $P \in G(n+\ell,\ell)$, $P^\perp \cong \real^n$, and, with respect to that identification, $Q \mapsto P \oplus Q$ defines an equivariant map from $G(n,k)$ into $G(n + \ell, k + \ell)$. (That is, the map in the proof of Proposition \ref{stiefel monotonicity}(c) descends to Grassmannians.)
\end{proof}

\noindent As with those of Stiefel manifolds, the indices and coindex of $G(n+m,k+\ell)$ are all, respectively, no smaller than those of $G(n,k)$ when $0 \leq \ell \leq m$. Since $G(n,0) \cong S^0$, the indices and coindex of $G(n,0)$ and $G(n,n)$ are all zero; similarly, since $G(n,1) \cong S^{n-1}$, those of $G(n,1)$ and $G(n,n-1)$ are all $n-1$.

\begin{proposition}\label{G(4,2)}
	The Conner--Floyd and Yang indices and coindex of $G(4,2)$ all equal two.
\end{proposition}

\begin{proof}
	It is widely known that $G(4,2)$ is homeomorphic to $S^2 \times S^2$; an elementary proof of this is given in \cite{Baralic2011}. Following the computations there, one finds that this homeomorphism is equivariant, where the action on $S^2 \times S^2$ is the product of the action on each factor. Since the product of the identity map on $S^2$ is equivariant, $\mathrm{ind}(G(4,2)) \geq 2$. Since projection onto either factor is equivariant, $\mathrm{coind}(G(4,2)) \leq 2$. 
\end{proof}

\noindent It's reasonable to wonder whether $\nu(G(n,k))$ can differ from $\nu(St(n,k))$. In fact, this happens when $(n,k) = (4,3)$.

\begin{proposition}\label{St(4,3)}
Each of the following holds:\\
\noindent \textbf{(a)} $\nu(St(4,3)) = 1 \textrm{ or } 2$\\
\noindent \textbf{(b)} $\nu(G(4,3)) = 3$
\end{proposition}

\begin{proof}
\noindent \textbf{(a)} Since $\nu(St(2,1)) = 1$, Proposition \ref{stiefel monotonicity}(c) implies that $\nu(St(4,3)) \geq 1$. By Proposition \ref{G(4,2)}, $\nu(G(4,2)) = 2$, so $\nu(St(4,2)) \leq 2$. The claim now follows from Proposition \ref{stiefel monotonicity}(a).\\
\noindent \textbf{(b)} As already noted, this may be seen by combining $\nu(G(4,1)) = 3$ and Proposition \ref{grassmannian monotonicity}(a).
\end{proof}

\noindent By employing similar reasoning, one may show that, for $n \leq 9$, the values of $\nu(G(n,k))$ and $\nu(St(n,k))$ must be among the entries of the following tables. The notation $\ell$:$m$ means that all integers from $\ell$ to $m$ are included.

\begin{center}
	\begin{table}\label{Grassmannian table}
	\caption{Possible values of $\nu(G(n,k))$}
	\begin{tabular}{c|c|cccccccccc}
		\hline
		\multicolumn{2}{c|}{\multirow{2}{*}{}} & \multicolumn{10}{c}{$k$} \\
		\cline{3-12}
		\multicolumn{2}{c|}{} & 0 & 1 & 2 & 3 & 4 & 5 & 6 & 7 & 8 & 9 \\
		\hline
		\multirow{10}{*}{$n$} & 0 & 0 & & & & & & & & & \\
		& 1 & 0 & 0 & & & & & & & & \\
		& 2 & 0 & 1 & 0 & & & & & & & \\
		& 3 & 0 & 2 & 2 & 0 & & & & & & \\
		& 4 & 0 & 3 & 2 & 3 & 0 & & & & & \\
		& 5 & 0 & 4 & 4:6 & 4:6 & 4 & 0 & & & & \\
		& 6 & 0 & 5 & 4:8 & 4:9 & 4:8 & 5 & 0 & & & \\
		& 7 & 0 & 6 & 6:10 & 4:12 & 4:12 & 6:10 & 6 & 0 & & \\
		& 8 & 0 & 7 & 6:12 & 6:15 & 4:16 & 6:15 & 6:12 & 7 & 0 & \\
		& 9 & 0 & 8 & 8:14 & 6:18 & 6:20 & 6:20 & 6:18 & 8:14 & 8 & 0 \\
		\hline
	\end{tabular}
	\end{table}
\end{center}

\vspace{2pt}

\begin{center}
	\begin{table}\label{Stiefel table}
	\caption{Possible values of $\nu(St(n,k))$}
	\begin{tabular}{c|c|ccccccccc}
		\hline
		\multicolumn{2}{c|}{\multirow{2}{*}{}} & \multicolumn{9}{c}{$k$} \\
		\cline{3-11}
		\multicolumn{2}{c|}{} & 1 & 2 & 3 & 4 & 5 & 6 & 7 & 8 & 9 \\
		\hline
		\multirow{9}{*}{$n$} & 1 & 0 & & & & & & & & \\
		& 2 & 1 & 0 & & & & & & & \\
		& 3 & 2 & 1:2 & 0 & & & & & & \\
		& 4 & 3 & 2 & 1:2 & 0 & & & & & \\
		& 5 & 4 & 3:4 & 2:4 & 1:4 & 0 & & & & \\
		& 6 & 5 & 4:5 & 3:5 & 2:5 & 1:5 & 0 & & & \\
		& 7 & 6 & 5:6 & 4:6 & 3:6 & 2:6 & 1:6 & 0 & & \\
		& 8 & 7 & 6 & 5:6 & 4:6 & 3:6 & 2:6 & 1:6 & 0 & \\
		& 9 & 8 & 7:8 & 6:8 & 5:8 & 4:8 & 3:8 & 2:8 & 1:8 & 0 \\
		\hline
	\end{tabular}
	\end{table}
\end{center}

\bibliographystyle{amsplain}
\bibliography{bibliography}

\end{document}